\documentclass[11pt, letterpaper,final]{amsart}
\usepackage{amsfonts,amsmath, amsthm, amssymb, latexsym, epsfig}
\usepackage[english]{babel}
\usepackage[utf8]{inputenc}
\usepackage[all]{xy}
\usepackage{xspace}
\usepackage{comment}
\usepackage{setspace}
\usepackage{enumerate}
\usepackage{stmaryrd}
\usepackage[mathscr]{eucal}

\usepackage[notcite,notref]{showkeys}

	\newcommand{\Prim}{\ensuremath{\operatorname{Prim}}}

	\theoremstyle{plain}
	\newtheorem{thm}{Theorem}
	\newtheorem{theorem}[thm]{Theorem}
	\newtheorem{corollary}[thm]{Corollary}
	
	\newtheorem{lemma}{Lemma}
	\newtheorem{proposition}[lemma]{Proposition}
	\newtheorem{theorem2}[lemma]{Theorem}
	\newtheorem{corollary2}[lemma]{Corollary}
	
	\theoremstyle{definition}
	\newtheorem{remark}[lemma]{Remark}

\begin{document}
	\title{Traceless AF embeddings and unsuspended $E$-theory}
	\author{James Gabe}
	\address{Current address: School of Mathematics and Applied Statistics, University of Wollongong, Wollongong NSW 2522, Australia.}
	\address{School of Mathematics and Statistics, University of Glasgow, University Place, Glasgow G12 8SQ, Scotland.}
        \email{jamiegabe123@hotmail.com}
	\thanks{This work was funded by the Carlsberg Foundation through an Internationalisation Fellowship.}

\begin{abstract}
I show that quasidiagonality and AF embeddability are equivalent properties for traceless $\mathrm C^\ast$-algebras and are characterised in terms of the primitive ideal space. For nuclear $\mathrm C^\ast$-algebras the same characterisation determines when Connes and Higson's $E$-theory can be unsuspended.
\end{abstract}

\maketitle

\section{Introduction}

Quasidiagonality is a fundamental concept of finiteness for $\mathrm C^\ast$-algebras with a  tight connection to the structure of nuclear $\mathrm C^\ast$-algebras, the Elliott classification programme, and amenability of discrete groups, see \cite{TikuisisWhiteWinter-QDnuc}. Voiculescu's influential work on quasi\-diagonality \cite{Voiculescu-qdhtpy} had the important consequence that for any $\mathrm C^\ast$-algebra $A$, the suspension $C_0(\mathbb R, A)$ and the cone $C_0((0,1] , A)$ of $A$ are both quasi\-diagonal.

By Connes' groundbreaking work \cite{Connes-class} it follows that a von Neumann algebra with a separable predual admits an embedding into the hyperfinite II$_1$-factor $\mathcal R$ if and only if it is injective with a faithful, normal trace. The $\mathrm C^\ast$-algebraic analogue is embeddability into an approximately finite dimensional $\mathrm C^\ast$-algebra, known as AF embeddability, which was first studied by Pimsner and Voiculescu \cite{PimsnerVoiculescu-embedding}.  
An immediate consequence of being AF embeddable is that the $\mathrm C^\ast$-algebra is separable, exact and quasi\-diagonal, and it is a major open problem whether the converse holds. This has been verified in many concrete cases such as for certain $\mathrm C^\ast$-dynamical constructions \cite{Pimsner-embtransgroup}, \cite{Brown-AFemb}, and for all reduced group $\mathrm C^\ast$-algebras \cite{OzawaRordamSato-elementaryamenable}, \cite{TikuisisWhiteWinter-QDnuc}.

Ozawa \cite{Ozawa-AFembeddability} proved an analogue of Voiculescu's quasidiagonality result and showed that the suspension and the cone of any separable, exact $\mathrm C^\ast$-algebra is AF embeddable. A surprising consequence is that many $\mathrm C^\ast$-algebras of a very infinite nature, such as $C_0(\mathbb R, \mathcal O_2)$, turn out to have the very finite property of being AF embeddable.

The suspension and the cone of any $\mathrm C^\ast$-algebra have an important property in common; their primitive ideal spaces have no non-empty, compact, open subsets. In fact, for $X = \mathbb R$ or $X=(0,1]$ any compact, open subset of $\Prim C_0(X, A) \cong  X \times \Prim A$ can be projected onto a compact, open subset of $X$ and must therefore be empty.

I show that a separable, exact $\mathrm C^\ast$-algebra is AF embeddable provided its primitive ideal space has no non-empty, compact, open subsets (Corollary \ref{c:AFemb}), thus generalising the above result of Ozawa. Consequently, in the absence of traces being AF embeddable, quasidiagonal, and stably finite are all equivalent conditions characterised by this property of the primitive ideal space (Corollary \ref{c:traceless}).

In celebrated work of Connes and Higson \cite{ConnesHigson-deformations} the semigroup $[[A, B\otimes \mathbb K]]$ of homotopy classes of asymptotic morphisms between separable $\mathrm C^\ast$-algebras was introduced. They showed that $E(A,B):=[[ SA, SB\otimes \mathbb K]]$ is an abelian group  which is naturally isomorphic to $KK(A,B)$ whenever $A$ is nuclear. Here $S = C_0(\mathbb R,-)$ denotes the suspension functor.

Dadarlat and Loring \cite{DadarlatLoring-unsuspend} initiated the study of when the suspension functor induces an isomorphism $[[A , B\otimes \mathbb K]] \to E(A,B)$, or in other words, when $E$-theory can be unsuspended. 
As a final result, building on recent results of Dadarlat and Pennig \cite{DadarlatPennig-defnil}, I show that for a separable, nuclear $\mathrm C^\ast$-algebra $A$ the property of having no non-empty, compact, open subsets in the primitive ideal space characterises exactly when $E(A,B)$ can be unsuspended for any separable $\mathrm C^\ast$-algebra $B$ (Corollary \ref{c:unsuspend}). 

\subsection{Strategy and the main results}

The fundamental $\mathcal O_2$-embedding theorem of Kirchberg \cite{Kirchberg-ICM}, \cite{KirchbergPhillips-embedding} states that any separable, exact $\mathrm C^\ast$-algebra admits an embedding into the Cuntz algebra $\mathcal O_2$. An important consequence of this theorem is the Kirchberg--Phillips classification \cite{Kirchberg-simple}, \cite{Phillips-classification} of separable, nuclear, purely infinite, simple $\mathrm C^\ast$-algebras. Kirchberg \cite{Kirchberg-non-simple} proved a far reaching generalisation of this theorem by classifying all separable, nuclear, $\mathcal O_\infty$-stable $\mathrm C^\ast$-algebras. 
A crucial intermediate result for proving this generalised classification result was an ideal related version of the $\mathcal O_2$-embedding theorem.

Recently in \cite{Gabe-O2class} I proved a slightly different version of this embedding theorem. 
Consequently, whether or not a separable, exact $\mathrm C^\ast$-algebra $A$ embeds into a separable, nuclear, $\mathcal O_\infty$-stable $\mathrm C^\ast$-algebra $B$ can be determined solely by studying the primitive ideal spaces of $A$ and $B$.

The strategy for proving the main results of this paper is to use this embedding theorem for a particular $\mathrm C^\ast$-algebra $B = \mathcal A_{[0,1]}$ introduced and studied by Rørdam in \cite{Rordam-purelyinfAH}. The $\mathrm C^\ast$-algebra $\mathcal A_{[0,1]}$ is an inductive limit of the $\mathrm C^\ast$-algebras $C_0((0,1], M_{2^n})$ and is in particular an ASH-algebra, i.e.~an inductive limit of $\mathrm C^\ast$-subalgebras of matrices over commutative, separable $\mathrm C^\ast$-algebras. Astonishingly, by appealing to deep structural results of Kirchberg and Rørdam from \cite{KirchbergRordam-absorbingOinfty}, Rørdam showed that $\mathcal A_{[0,1]}$ is $\mathcal O_\infty$-stable. Combining this with the ideal related $\mathcal O_2$-embedding theorem one obtains the following main theorem. 

\begin{theorem}\label{t:A}
Let $A$ be a separable, exact $\mathrm C^\ast$-algebra. The following are equivalent.
\begin{itemize}
\item[$(i)$] $A$ embeds into Rørdam's purely infinite ASH-algebra $\mathcal A_{[0,1]}$;
\item[$(ii)$] $A$ embeds into $C_0((0,1],\mathcal O_2)$;
\item[$(iii)$] $A$ embeds into a zero-homotopic $\mathrm C^\ast$-algebra;
\item[$(iv)$] the primitive ideal space $\Prim A$ has no non-empty, compact, open subsets.
\end{itemize}
\end{theorem}

The proof of the above main theorem is postponed until Section \ref{s:A}. Below are some corollaries of the theorem. 

As ASH-algebras are AF embeddable, see for instance \cite[Proposition 4.1]{Rordam-purelyinfAH}, the following corollary is immediate.

\begin{corollary}\label{c:AFemb}
Let $A$ be a separable, exact $\mathrm C^\ast$-algebra for which $\Prim A$ has no non-empty, compact, open subsets. Then $A$ is AF embeddable (and in particular quasidiagonal).
\end{corollary}

An \emph{extended trace} $\tau \colon A_+ \to [0,\infty]$ is an additive, homogeneous map such that $\tau(a^\ast a) = \tau(aa^\ast)$ for all $a\in A$. An exact $\mathrm C^\ast$-algebra is called \emph{traceless} if all lower semicontinuous extended traces on $A$ only take the values $0$ and $\infty$. Due to Ozawa's result \cite{Ozawa-AFembeddability} it is known that many traceless $\mathrm C^\ast$-algebras are AF embeddable such as $C_0(\mathbb R,\mathcal O_2)$. Here is a complete characterisation of when this is the case. The proof is presented in Section \ref{s:traceless}.

\begin{corollary}\label{c:traceless}
Let $A$ be a separable, exact, traceless $\mathrm C^\ast$-algebra. The following are equivalent.
\begin{itemize}
\item[$(i)$] $A$ is AF embeddable;
\item[$(ii)$] $A$ is quasidiagonal;
\item[$(iii)$] $A$ is stably finite;
\item[$(iv)$] $\Prim A$ has no non-empty, compact, open subsets.
\end{itemize}
\end{corollary}

The Blackadar--Kirchberg problem \cite[Question 7.3.1]{BlackadarKirchberg-genindlimfindim}, also known as the quasidiagonality question or abbreviated QDQ, asks whether any separable, nuclear, stably finite $\mathrm C^\ast$-algebra is quasidiagonal. Corollary \ref{c:AFemb} gives an affirmative answer for $\mathrm C^\ast$-algebras with no non-empty, compact, open subsets in their primitive ideal spaces. As such $\mathrm C^\ast$-algebras are always stably projectionless they are automatically stably finite.\footnote{Recall that a not necessarily unital $\mathrm C^\ast$-algebra $A$ is \emph{stably finite} if for every $n\in \mathbb N$, $M_n(\widetilde A)$ has no non-unitary isometries. In particular, stably projectionless $\mathrm C^\ast$-algebras are stably finite.} Combined with Corollary \ref{c:traceless} for the traceless case one immediately obtains the following.

\begin{corollary}
The Blackadar--Kirchberg problem has an affirmative answer for all traceless $\mathrm C^\ast$-algebras, and for all $\mathrm C^\ast$-algebras for which the primitive ideal space has no non-empty, compact, open subsets.
\end{corollary}

 The final corollary, which is proved in Section \ref{s:unsuspended}, shows that unsuspendability of Connes and Higson's $E$-theory is a spectral property for nuclear $\mathrm C^\ast$-algebras. Recall that for separable, nuclear $\mathrm C^\ast$-algebras $A$, $E(A,B)$ and $KK(A,B)$ are naturally isomorphic for any separable $\mathrm C^\ast$-algebra $B$ so the result below gives a characterisation of certain $KK$-groups.

\begin{corollary}\label{c:unsuspend}
Let $A$ be a separable, nuclear $\mathrm C^\ast$-algebra. The following are equivalent.
\begin{itemize}
\item[$(i)$] $[[ A, B \otimes \mathbb K]]$ is a group for any separable $\mathrm C^\ast$-algebra $B$; 
\item[$(ii)$] the suspension functor induces a natural isomorphism 
\begin{equation}
[[A, B\otimes \mathbb K]] \xrightarrow \cong E(A,B)
\end{equation}
 for any separable $\mathrm C^\ast$-algebra $B$;
\item[$(iii)$] $\Prim A$ has no non-empty, compact, open subsets.
\end{itemize}
\end{corollary}

\section{Proof of Theorem \ref{t:A}}\label{s:A}

Recall that a \emph{complete lattice} $\mathcal L$ is a partially ordered set in which every subset has a supremum and an infimum. If $x,y\in \mathcal L$ then $x$ is said to be \emph{compactly contained} in $y$ (or way below $y$), written as $x \Subset y$, if whenever $(y_i)$ is an increasing net in $\mathcal L$ such that $y\leq \sup y_i$, then $x \leq y_i$ for some $i$. An element $x\in \mathcal L$ is \emph{compact} if $x \Subset x$.

Let $\mathcal I(A)$ denote the \emph{ideal lattice} of a $\mathrm C^\ast$-algebra $A$, i.e.~the partially ordered set of two-sided, closed ideals in $A$. It is well-known that $\mathcal I(A)$ is a complete lattice with $\inf I_\alpha = \bigcap I_\alpha$ and $\sup I_\alpha = \overline{\sum I_\alpha}$.

\begin{remark}\label{r:1}
For any $\mathrm C^\ast$-algebra $A$ there is an order isomorphism between $\mathcal I(A)$ and the set of open subsets of $\Prim A$, see for instance \cite[Theorem 4.1.3]{Pedersen-book-automorphism}. Clearly an ideal $I\in \mathcal I(A)$ is compact if and only if it corresponds to a compact, open set in $\Prim A$. In particular, $\Prim A$ has no non-empty, compact, open subsets if and only if $\mathcal I(A)$ has no non-zero, compact elements.
\end{remark}

A partially ordered set $\mathcal L$ is called a \emph{continuous lattice} if $\mathcal L$ is a complete lattice such that any $y \in \mathcal L$ is the supremum of the set $\{ x \in \mathcal L : x \Subset y\}$.

A \emph{basis} $\mathcal B$ of a continuous lattice $\mathcal L$ is a subset such that any $y\in \mathcal L$ the set $\{ x \in \mathcal B: x \Subset y\}$ is upwards directed and has supremum $y$.

\begin{proposition}\label{p:1}
$\mathcal I(A)$ is a continuous lattice for any $\mathrm C^\ast$-algebra $A$, and if $A$ is separable then $\mathcal I(A)$ has a countable basis as a continuous lattice.
\end{proposition}
\begin{proof}
For any $I\in \mathcal I(A)$, one clearly has
\begin{equation}
I = \overline{\sum_{x\in I_+ , \epsilon >0} \overline{Af_\epsilon (x)A}} = \sup_{x\in I_+ , \epsilon >0}  \overline{Af_\epsilon (x)A},
\end{equation}
where $f_\epsilon \colon [0,\infty) \to [0,\infty)$ is the continuous function given by $f_\epsilon(t) = \max\{ 0, t-\epsilon\}$. By \cite[Lemma 2.2]{Gabe-O2class} one has $ \overline{Af_\epsilon (x)A} \Subset I$ for any $x\in I_+$ and $\epsilon >0$, so $\mathcal I(A)$ is a continuous lattice.

By \cite[Corollary 4.3.4]{Pedersen-book-automorphism}, $\mathcal I(A)$ has a countable basis $\mathcal B$ as a complete lattice, i.e.~for every $I\in \mathcal I(A)$ there is a subset $\mathcal S \subseteq \mathcal B$ such that $I = \overline{\sum_{J\in \mathcal S} J}$. By replacing $\mathcal B$ by its finite join closure (which is still countable) we may assume without loss of generality that $\mathcal B$ has finite joins.

Then $\mathcal B$ is a basis for $\mathcal I(A)$ as a continuous lattice. To see this let $I\in \mathcal I(A)$. By \cite[Corollary 2.3]{Gabe-O2class} there is a sequence $I_1 \Subset I_2 \Subset \dots$ in $\mathcal I(A)$ such that $I = \overline{\bigcup I_n}$. As $\mathcal B$ is an upwards directed basis we may by compact containment find $J_n \in \mathcal B$ such that $I_{2n} \subseteq J_n \subseteq I_{2n+1}$. Clearly $J_1 \Subset J_2 \Subset \dots$ and $I =\overline{\bigcup J_n}$. Moreover, $\{ J \in \mathcal B : J \Subset I\}$ is upwards directed since $\mathcal B$ is closed under finite joins and since finite joins preserve $\Subset$.
\end{proof}

Let $\mathcal L$ and $\mathcal L'$ be complete lattices and $\Phi \colon \mathcal L \to \mathcal L'$ be an order preserving map. Following \cite[Definition 2.6 and Remark 2.8]{Gabe-O2class}, $\Phi$ is called a \emph{$Cu$-morphism} if it preserves arbitrary suprema and compact containment.

The proof of the following lemma uses several non-trivial results about continuous lattices. The results are by now considered classical and can be found in the text book \cite{GHKLMS-book}. I will make appropriate citations when needed, as to make the proof as readable as possible for non-experts.

In what follows, $[0,1]$ denotes the unit interval with the usual order and topology.

\begin{lemma}\label{l:3}
Let $A$ be a separable $\mathrm C^\ast$-algebra such that the primitive ideal space $\Prim A$ has no non-empty, compact, open subsets. Then there exists a $Cu$-morphism $\Phi \colon \mathcal I(A) \to [0,1]$ such that $\Phi^{-1}(\{0\}) = \{ 0\}$.
\end{lemma}
\begin{proof}
$\mathcal I(A)$ is a continuous lattice by Proposition \ref{p:1}. The fundamental theorem of compact semilattices \cite[Theorem VI-3.4]{GHKLMS-book} implies that $\mathcal I(A)$ equipped with the so-called Lawson topology is a compact topological semilattice with small semilattices.\footnote{One does not need to know how the Lawson topology is defined, nor what having small semilattices means, in order to understand the proof.} Here compact topological semilattice means that the topology on $\mathcal I(A)$ is compact, Hausdorff, and that the meet map $\mathcal I(A) \times \mathcal I(A) \to \mathcal I(A)$, $(I,J) \mapsto I\cap J$ is continuous, see \cite[Definition VI-1.11]{GHKLMS-book}.  

\emph{Claim:} There is a continuous, order preserving map $\Psi \colon [0,1] \to \mathcal I(A)$ such that $\Psi(0) = 0$ and $\Psi(1) = A$.

Suppose the claim has been proved. Clearly any order preserving map out of $[0,1]$ is a semilattice homomorphism, i.e.~a map that preserves meets. As $\Psi$ is continuous and $\Psi(1) = A$, the fundamental theorem of compact semilattices (cited above) implies that $\Psi$ has a lower adjoint $\Phi \colon \mathcal I(A) \to [0,1]$ which preserves compact containment. Being a lower adjoint means that $\Phi$ is order preserving and for $(I, t) \in \mathcal I(A)\times [0,1]$ one has $I \subseteq \Psi(t)$ if and only if $\Phi(I)  \leq t$, see \cite[Definition O-3.1]{GHKLMS-book}. By \cite[Theorem O-3.3]{GHKLMS-book} $\Phi$ preserves suprema so $\Phi$ is a $Cu$-morphism. By the definition of lower adjoints, $\Phi(I) = 0$ implies that $I \subseteq \Psi(0) = 0$. Consequently, $\Phi^{-1}(\{0\}) = \{0\}$ and thus it remains to prove the above claim.

As $\mathcal I(A)$ has a countable basis as a continuous lattice by Proposition \ref{p:1}, the Lawson topology on $\mathcal I(A)$ is second countable by \cite[Theorem III-4.5]{GHKLMS-book}.\footnote{The theorem is applicable for domains and thus also for continuous lattices since a continuous lattice is a domain which is also a complete lattice, see \cite[Defintion I-1.6]{GHKLMS-book}.} Thus $\mathcal I(A)$ is compact and metrisable in the Lawson topology. 

By Remark \ref{r:1}, $\mathcal I(A)$ has no non-zero, compact elements. By \cite[Theorem VI-5.11]{GHKLMS-book} it follows that the elements $0,A\in \mathcal I(A)$ both lie on an arc chain, i.e.~there is a subset $\mathcal C \subseteq \mathcal I(A)$ with $0,A \in \mathcal C$ for which $\mathcal C$ is a compact, connected chain, see \cite[Definition VI-5.5]{GHKLMS-book}.\footnote{Arc chains are defined as subsets of pospaces. By \cite[Proposition VI-1.14]{GHKLMS-book} it follows that any compact semilattice is a pospace, so this makes sense in our case.} As $\mathcal I(A)$ is metrisable, so is $\mathcal C$.

By \cite[Proposition VI-5.4]{GHKLMS-book},\footnote{This is applicable since the order on $\mathcal I(A)$ is semiclosed with respect to the Lawson topology due to the remark following \cite[Definition III-5.1]{GHKLMS-book}.} the induced topology on $\mathcal C$ is the order topology, so $\mathcal C$ is a compact, metrisable, connected chain with exactly two non-cut points.\footnote{In fact, if $0 < I < A$ in $\mathcal C$, then $[0,I)$ and $(I,A]$ are both clopen subsets of $\mathcal C\setminus\{I\}$ in the order topology, so any such $I$ is a cut point. The elements $0$ and $A$ are not cut-points as these are the minimal and maximal element respectively in $\mathcal C$.} 
Hence there is an order preserving homeomorphism $\Psi \colon [0,1] \to \mathcal C$ which necessarily satisfies $\Psi(0) = 0$ and $\Psi(1) = A$. 
The induced map $\Psi \colon [0,1] \to \mathcal C \subseteq \mathcal I(A)$ is continuous, order preserving and $\Psi(0) = 0$, $\Psi(1) = A$. This finishes the proof of the claim and thus also of the lemma.
\end{proof}

If $\phi \colon A \to B$ is a $\ast$-homomorphism between $\mathrm C^\ast$-algebras then 
\begin{equation}
\mathcal I(\phi) \colon \mathcal I(A) \to \mathcal I(B), \qquad I \mapsto \overline{B\phi(I)B}
\end{equation}
is a $Cu$-morphism by \cite[Lemma 2.12]{Gabe-O2class}. 
A major tool in the proof of Theorem \ref{t:A} is the ideal related version of Kirchberg’s $\mathcal O_2$-embedding theorem \cite[Hauptsatz 2.15]{Kirchberg-non-simple}, in the form given in \cite{Gabe-O2class}.

\begin{theorem2}[{\cite[Theorem 6.1]{Gabe-O2class}}]\label{t:O2}
Let $A$ be a separable, exact $\mathrm C^\ast$-algebra, let $B$ be a separable, nuclear, $\mathcal O_\infty$-stable $\mathrm C^\ast$-algebra, and let $\Phi \colon \mathcal I(A) \to \mathcal I(B)$ be a $Cu$-morphism. Then there exists a $\ast$-homomorphism $\phi \colon A \to B$ such that $\mathcal I(\phi) = \Phi$.
\end{theorem2}

Rørdam \cite{Rordam-purelyinfAH} constructed the $\mathrm C^\ast$-algebra $\mathcal A_{[0,1]}$ as follows: Let $(t_n)_{n\in \mathbb N}$ be a dense sequence in $[0,1)$ and let $\phi_n \colon C_0([0,1), M_{2^{n}}) \to C_0([0,1), M_{2^{n+1}})$ be given by
\begin{equation}
\phi_n(f) (t) = \left( \begin{array}{cc} f(t) & 0 \\ 0 &  f(\max\{ t, t_n\}) \end{array} \right) \in M_2(M_{2^n}) \cong M_{2^{n+1}}
\end{equation}
for $f\in C_0([0,1),M_{2^n})$, $t\in [0,1)$. Define 
\begin{equation}\label{eq:A01}
\mathcal A_{[0,1]} := \varinjlim (C_0([0,1), M_{2^n}), \phi_n).
\end{equation}
It follows from Kirchberg’s classification of non-simple $\mathcal O_\infty$-stable $\mathrm C^\ast$-algebras that the isomorphism class of $\mathcal A_{[0,1]}$ does not depend on the choice of $(t_n)_{n\in \mathbb N}$ (see \cite{KirchbergRordam-zero}) though we do not need this fact here.

The essential properties of $\mathcal A_{[0,1]}$ required for Theorem \ref{t:A} are summerised in the next result which is a combination of \cite[Proposition 2.1 and Corollary 5.3]{Rordam-purelyinfAH}.

\begin{theorem2}[Rørdam]\label{t:Rordam}
The $\mathrm C^\ast$-algebra $\mathcal A_{[0,1]}$ is separable, nuclear, $\mathcal O_\infty$-stable, and $ \mathcal I(\mathcal A_{[0,1]}) \cong [0,1]$.
\end{theorem2}

\begin{proof}[Proof of Theorem \ref{t:A}]
$(i)\Rightarrow (ii)$: It is enough to show that $\mathcal A_{[0,1]}$ embeds into $C_0((0,1],\mathcal O_2)$. By \cite[Proposition 6.1]{KirchbergRordam-zero}, $\mathcal A_{[0,1]}$ is zero-homotopic so in particular $\mathcal A_{[0,1]}$ embeds into $C_0((0,1],\mathcal A_{[0,1]})$. By the $\mathcal O_2$-embedding theorem \cite{KirchbergPhillips-embedding}, this $\mathrm C^\ast$-algebra embeds into $C_0((0,1],\mathcal O_2)$.\footnote{Alternative proof: The map $\Phi \colon [0,1] \to \mathcal I(C_0((0,1],\mathcal O_2))$ given by $\Phi(t) = C_0((\tfrac{1-t}{2}, \tfrac{1+t}{2}), \mathcal O_2)$ is a $Cu$-morphism for which $\Phi^{-1}(\{0\}) = \{ 0\}$. Hence $\Phi$ induces an embedding $\mathcal A_{[0,1]} \hookrightarrow C_0((0,1],\mathcal O_2)$ by Theorems \ref{t:O2} and \ref{t:Rordam}.} 

$(ii)\Rightarrow (iii)$: Follows since $C_0((0,1],\mathcal O_2)$ is zero-homotopic.

$(iii)\Rightarrow (iv)$: Suppose $D$ is zero-homotopic and that $A$ embeds into $D$. As $D$ embeds into $C_0((0,1], D)$, so does $A$. Fix an embedding $\phi \colon A \to C_0((0,1],D)$. Suppose for contradiction that $\Prim A$ has a non-empty, compact, open subset, or equivalently by Remark \ref{r:1}, that $A$ contains a non-zero compact ideal $I$. As $\mathcal I(\phi) \colon \mathcal I(A) \to \mathcal I(C_0((0,1],D))$ is a $Cu$-morphism and since $\phi$ is injective, $\mathcal I(\phi)(I)$ is a non-zero, compact ideal in $C_0((0,1], D)$. Hence $\Prim(C_0((0,1],D)) \cong (0,1]\times \Prim D$ contains a non-empty, compact, open subset. By projecting onto $(0,1]$, we obtain a non-empty, compact, open subset of $(0,1]$ which is a contradiction. Hence $(iv)$ follows.

$(iv)\Rightarrow (i)$: Suppose that $\Prim A$ has no non-empty, compact, open subsets. By Lemma \ref{l:3} there is a $Cu$-morphism $\Phi \colon \mathcal I(A) \to [0,1]$ such that $\Phi^{-1}(\{0\}) = \{0\}$.  By Theorem \ref{t:Rordam}, $\mathcal A_{[0,1]}$ is separable, nuclear, $\mathcal O_\infty$-stable and there is an order isomorphism $\Theta \colon [0,1] \xrightarrow \cong \mathcal I(\mathcal A_{[0,1]})$. So by Theorem \ref{t:O2} there is a $\ast$-homomorphism $\phi \colon A \to \mathcal A_{[0,1]}$ such that $\mathcal I(\phi) = \Theta \circ \Phi$. As $\mathcal I(\phi)^{-1}(\{0\}) = \{ 0 \}$ it follows that $\phi$ is injective.
\end{proof}

\section{Proof of Corollary \ref{c:traceless}}\label{s:traceless}

Recall that a separable $\mathrm C^\ast$-algebra $B$ is said to be \emph{MF} if there is a sequence $(k_i)_{i\in \mathbb N}$ in $\mathbb N$ such that $B$ embeds into $ \prod_{i\in \mathbb N}M_{k_i} / \bigoplus_{i\in \mathbb N} M_{k_i}$. In particular, any separable, quasidiagonal $\mathrm C^\ast$-algebra is MF.

Before proving Corollary \ref{c:traceless}, I prove the following elementary lemma which is definitely well-known to many people.
In what follows, $\widetilde A$ denotes the minimal unitisation of $A$, and $\otimes$ denotes the spatial tensor product.

\begin{lemma}\label{l:stablyfinite}
Let $A$ be an exact, stably finite $\mathrm C^\ast$-algebra and let $B$ be a separable, MF $\mathrm C^\ast$-algebra. Then   $A \otimes B$ is stably finite.
\end{lemma}
\begin{proof}
Suppose for contradiction that $A \otimes B$ is not stably finite. Then
\begin{equation}
M_n(\widetilde{A \otimes B}) \subseteq M_n(\widetilde A \otimes \widetilde B) \cong \widetilde A \otimes M_n(\widetilde B)
\end{equation}
contains a non-unitary isometry for some $n$. Note that $M_n(\widetilde B)$ is also separable and MF. Any embedding $M_n(\widetilde B) \hookrightarrow \prod_{i\in \mathbb N}M_{k_i} / \bigoplus_{i\in \mathbb N} M_{k_i}$ induces
\begin{equation}
\widetilde A \otimes M_n(\widetilde B) \hookrightarrow \widetilde A \otimes \frac{\prod M_{k_i}}{\bigoplus M_{k_i}} \cong \frac{\widetilde A \otimes \prod M_{k_i}}{\widetilde A \otimes \bigoplus  M_{k_i} } \hookrightarrow \frac{\prod  \widetilde A \otimes M_{k_i}}{\bigoplus  \widetilde A \otimes M_{k_i}},
\end{equation}
where the isomorphism above exists by exactness of $\widetilde A$. Thus $\frac{\prod  \widetilde A \otimes M_{k_i}}{\bigoplus  \widetilde A \otimes M_{k_i}}$ contains a non-unitary isometry $v$. Let $(\widetilde v_i)_{i\in \mathbb N} \in \prod \widetilde A \otimes M_{k_i}$ be a lift of $v$. One has $\| \widetilde v_i^\ast \widetilde v_i - 1\| < 1$ for large $i$, so $\widetilde v_i^\ast \widetilde v_i$ is invertible. Let $v_i := \widetilde v_i (\widetilde v_i^\ast \widetilde v_i)^{-1/2}$ for sufficiently large $i$, and otherwise $v_i :=1$. Then $(v_i)_{i\in \mathbb N}$ is a lift of $v$, each $v_i$ is an isometry, and $\limsup_{i\to \infty} \| v_i v_i^\ast - 1\| = \| vv^\ast - 1\| = 1$. Hence $v_i \in \widetilde A \otimes M_{k_i}$ is a non-unitary isometry for some $i$. This contradicts that $A$ is stably finite, so $A\otimes B$ must be stably finite.
\end{proof}

\begin{proof}[Proof of Corollary \ref{c:traceless}]
$(iv)\Rightarrow (i)$ follows from Corollary \ref{c:AFemb}, and $(i) \Rightarrow (ii)\Rightarrow (iii)$ is well-known.

$(iii)\Rightarrow (iv)$: Suppose that $A$ is stably finite and let $\mathcal Z$ denote the Jiang--Su algebra. By \cite[Corollary 5.1]{Rordam-realrankZ}, $A\otimes \mathcal Z$ is purely infinite. As $\Prim A \cong \Prim A\otimes \mathcal Z$ (since $\mathcal Z$ is simple and nuclear) it follows from \cite[Proposition 2.7]{PasnicuRordam-purelyinfrr0} that any non-empty, compact, open subset in $\Prim A$ gives rise to a non-zero projection in $A\otimes \mathcal Z$. As any non-zero projection in a purely infinite $\mathrm C^\ast$-algebra is properly infinite, see \cite[Theorem 4.16]{KirchbergRordam-purelyinf}, and as $A\otimes \mathcal Z$ is stably finite by Lemma \ref{l:stablyfinite}, it follows that $\Prim A$ has no non-empty, compact, open subsets. Hence $(iii)\Rightarrow (iv)$.
\end{proof}

\section{Proof of Corollary \ref{c:unsuspend}}\label{s:unsuspended}

Following Dadarlat and Loring \cite{DadarlatLoring-unsuspend} a separable $\mathrm C^\ast$-algebra $A$ is called \emph{homotopy symmetric} if $[[A, A\otimes \mathbb K]]$ is a group. They proved the stunning result, that if $A$ is homotopy symmetric then the unsuspended $E$-theory $[[A,B\otimes \mathbb K]]$ is a group for \emph{any} separable $\mathrm C^\ast$-algebra $B$, and moreover that the suspension functor induces an isomorphism 
\begin{equation}
[[A, B\otimes \mathbb K]] \xrightarrow \cong E(A,B).
\end{equation}
While homotopy symmetry characterises when $E$-theory is unsuspendable, it is in general a property which is very hard to verify. Dadarlat and Pennig \cite{DadarlatPennig-connective} introduced connectivity of $\mathrm C^\ast$-algebras to gain a better understanding of homotopy symmetry. A separable $\mathrm C^\ast$-algebra $A$ is \emph{connective}\footnote{In \cite{DadarlatPennig-defnil} the same definition was called \emph{property (QH)}.} if there is a embedding 
\begin{equation}
A \to \frac{\prod_{\mathbb N} C_0((0,1], B(\mathcal H))}{\bigoplus_{\mathbb N} C_0((0,1], B(\mathcal H))}
\end{equation}
which has a completely positive contractive lift $A \to \prod_{\mathbb N} C_0((0,1], B(\mathcal H))$. Here $\mathcal H$ is a separable, infinite dimensional Hilbert space.
Dadarlat and Pennig proved the following remarkable theorem \cite[Theorem 3.1]{DadarlatPennig-defnil}. 

\begin{theorem2}[Dadarlat--Pennig]\label{t:DP}
A separable, nuclear $\mathrm C^\ast$-algebra is homotopy symmetric if and only if it is connective.
\end{theorem2}

Applying results from \cite{PasnicuRordam-purelyinfrr0}, it was observed in \cite[Proposition 2.7$(i)$]{DadarlatPennig-connective} that having a non-empty, compact, open subset in the primitive ideal space is an obstruction of connectivity. Using an embedding theorem of Blanchard for continuous fields \cite{Blanchard-O2cont}, it was shown in \cite[Proposition 2.7$(ii)$]{DadarlatPennig-connective} that this is the only obstruction for separable, nuclear $\mathrm C^\ast$-algebras provided the primitive ideal space is Hausdorff.

These results were my main inspiration for proving Theorem \ref{t:A}, which implies that this spectral obstruction is the only obstruction of connectivity for all separable, exact $\mathrm C^\ast$-algebras. 

\begin{corollary2}\label{c:connective}
A separable, exact $\mathrm C^\ast$-algebra is connective if and only if its primitive ideal space has no non-empty, compact, open subsets.
\end{corollary2}
\begin{proof}
``Only if'' is \cite[Proposition 2.7$(i)$]{DadarlatPennig-connective}. For ``if'', note that connectivity passes to $\mathrm C^\ast$-subalgebras so by Theorem \ref{t:A} it suffices to prove that $\mathcal A_{[0,1]}$ is connective. By construction, see \eqref{eq:A01}, $\mathcal A_{[0,1]}$ is an inductive limit of $\mathrm C^\ast$-algebras of the form $C_0([0,1), M_{2^n})$ which are connective. It follows from \cite[Theorem 3.3$(c)$]{DadarlatPennig-defnil} that $\mathcal A_{[0,1]}$ is connective.
\end{proof}

This corollary together with Theorem \ref{t:DP} and the main result of \cite{DadarlatLoring-unsuspend} gives the proof of Corollary \ref{c:unsuspend}. Although $(i)\Rightarrow (ii)$ is a well-known consequence of \cite[Theorem 4.3]{DadarlatLoring-unsuspend}, I will present a sketch of the proof for completion.

\begin{proof}[Proof of Corollary \ref{c:unsuspend}]
$(ii) \Rightarrow (i)$: This is obvious.

$(i) \Rightarrow (ii)$: Suppose $(i)$ holds. By \cite[Theorem 4.3]{DadarlatLoring-unsuspend}, the double-suspension $S\circ S$ induces an isomorphism $[[A, B \otimes \mathbb K]] \cong E(SA, SB)$ for all $B$ when $A$ is also assumed to be stable. It is well-known that stability of $A$ is not needed for this result to hold, and was only assumed to ease notation. Moreover, as $E$-theory satisfies Bott periodicity, suspension induces an isomorphism $E(A,B) \cong E(SA,SB)$ and thus it clearly follows that suspension induces an isomorphism $[[A, B\otimes \mathbb K]] \xrightarrow \cong E(A,B)$ for any $B$.

$(i) \Leftrightarrow (iii)$: As any nuclear $\mathrm C^\ast$-algebra is exact, this is an immediate consequence of Theorem \ref{t:DP} and Corollary \ref{c:connective}.
\end{proof}

\subsection*{Acknowledgment} I would like to thank the referee for suggesting that I expand Theorem \ref{t:A}. I am very grateful to Stuart White for several helpful comments and suggestions.

\newcommand{\etalchar}[1]{$^{#1}$}

\end{document}